\documentclass{jaca}
\usepackage{graphicx} %
\usepackage{tikz}
\usetikzlibrary{patterns}

\newtheorem{lemma}{Lemma}

\newcommand{\ie}{{\it i.e.}}
\newcommand{\m}{^{(m)}}
\newcommand{\nnn}{{n\in\xN}}
\newcommand{\mnn}{{m\in\xN}}

\newcommand{\bfC}{{\boldsymbol C}}
\newcommand{\bfI}{{\boldsymbol I}}

\newcommand{\bfH}{{\boldsymbol H}}

\newcommand{\bfu}{{\boldsymbol u}}

\newcommand{\bfe}{{\boldsymbol e}}
\newcommand{\bfn}{\boldsymbol n}
\newcommand{\bfx}{\boldsymbol x}

\newcommand{\bfvarphi}{{\boldsymbol \varphi}}

\newcommand{\dx}{\, \mathrm{d}\bfx}
\newcommand{\dt}{\, \mathrm{d} t}

\newcommand{\ui}{u_i}

\newcommand{\Hmeshzero}{\bfH_{\edges,0}}

\newcommand{\Hmeshi}{H_{\edges^{(i)}}}
\newcommand{\Hmeshizero}{H_{\edges^{(i)},0}}
\newcommand{\Hmeshunzero}{H_{\edges^{(1)},0}}
\newcommand{\Hmeshdezero}{H_{\edges^{(2)},0}}

\newcommand{\characteristic}{{1 \! \! 1}}
\newcommand{\dive}{{\rm div}}

\newcommand{\iinud}{i=1,2}

\newcommand{\nnud}{n \in \llbracket 0, N-1 \rrbracket}

\newcommand{\xN}{\mathbb{N}}
\newcommand{\xR}{\mathbb{R}}
\newcommand{\xL}{\mathrm{L}}

\newcommand{\mesh}{{\mathcal M}}
\newcommand{\edge}{{\sigma}}
\newcommand{\edgeperp}{{\tau}}
\newcommand{\edges}{{\mathcal E}}
\newcommand{\edgesK}{\edges(K)}

\newcommand{\edgesint}{{\mathcal E}_{\mathrm{int}}}
\newcommand{\edgesext}{{\mathcal E}_{\mathrm{ext}}}
\newcommand{\edgesinti}{{\mathcal E}_{\mathrm{int}}^{(i)}}

\newcommand{\edgesexti}{{\mathcal E}_{\mathrm{ext}}^{(i)}}
\newcommand{\edgesi}{{\edges\ei}}
\newcommand{\edgesj}{{\edges\ej}}
\newcommand{\edged}{\epsilon}
\newcommand{\edgesd}{{\widetilde {\edges}}}
\newcommand{\edgesdi}{{\edgesd^{(i)}}}
\newcommand{\edgesdinti}{{\edgesd^{(i)}_{{\rm int}}}}
\newcommand{\edgesdexti}{{\edgesd^{(i)}_{{\rm ext}}}}
\newcommand{\ei}{^{(i)}}
\newcommand{\ej}{^{(j)}}
\newcommand{\nKedge}{\bfn_{K,\edge}}

\begin{document}
\title{A decoupled staggered scheme for the shallow water equations}
\author{Rapha\`ele Herbin, Jean-Claude Latch\'e, \\ Youssouf Nasseri and Nicolas Therme}

\begin{abstract}
We present a first order scheme based on a staggered grid for the shallow water equations with topography in two space dimensions, which enjoys several properties: positivity of the water height, preservation of constant states, and weak consistency with the equations of the problem and with the associated entropy inequality.
\end{abstract}

\KeysAndCodes{Shallow water, finite volumes, staggered grid}{65M08,76B99}
%
%
\section{Introduction}\label{sec1:pbcont}

The shallow water equations form a hyperbolic system of two conservation equations (mass and momentum) which are obtained when modelling a flow whose vertical height is considered small with respect to the plane scale. 
The solution of such a system may develop shocks, so that the finite volume method is usually preferred for numerical simulations. 
Two main approaches are found: one is the colocated approach which is usually based on some approximate Riemann solver, see e.g. \cite{bou-04-non} and references therein; the other one is based on a staggered arrangement of the unknowns on the grid.
Indeed, staggered schemes have been used for some time in the hydraulic and ocean engineering community, see e.g. \cite{ara-81-pot,bon-05-ana,ste-03-sta}.
They have been recently analysed in the case of one space dimension \cite{doy-14-exp, gun-15-num}, following the works on the related barotropic Euler equations, see \cite{her-18-con} and references therein.
In the present work, we obtain a discrete local entropy inequality; furthermore, we extend the consistency analysis of the scheme to the case of two space dimensions, and we weaken the assumptions on the estimates, namely we no longer require a bound on the $BV$ norm of the approximate solutions, at least for the weak formulation (the passage to the limit in the entropy still necessitates a time $BV$ boundedness).
 
\medskip
Let $\Omega$ be an open bounded domain of $\xR^2$ and let $T >0$.
We consider the shallow water equations with topography over the space and time domain $\Omega \times (0,T)$:
\begin{subequations} \label{eq:sw}
\begin{align}\label{eq:mass} &
\partial_t h + \dive (h \bfu) =0 && \mbox{in } \; \Omega \times (0, T),
\\ \label{eq:mem} &
\partial_t (h \bfu) + \dive(h \bfu \otimes \bfu) + \nabla p + g h \nabla z = 0 && \mbox{in }\ \Omega \times (0, T), 
\\ \label{bc} &
p = \frac{1}{2} g h^2 && \mbox{in }\ \Omega \times (0, T) ,
\\ \label{bc:ins} &
\bfu \cdot \bfn =0 && \mbox{ on }\ \partial\Omega \times (0,T),
\\ \label{initiale} &
h(\bfx, 0) = h_0, \,\, \bfu(\bfx,0)=\bfu_0 && \mbox{ in} \ \Omega. 
\end{align} \end{subequations}
where $t$ stands for the time, $g$ is the standard gravity constant and $z$ the (given) topography, which is supposed to be regular in this paper.
These equations solve the water height $h$ and the velocity $\bfu$. 

\medskip
Let us recall that if $(h, \bfu)$ is a regular solution of \eqref{eq:sw}, the following elastic potential energy balance and kinetic energy balance is obtained by manipulations on the mass and momentum equations:
\begin{align} \label{pot_bal} &
\partial_t (\frac{1}{2} g h^2) + \dive (\frac{1}{2} g h^2 \bfu) + \frac{1}{2} g h^2 \dive \bfu = 0
\\ \label{kin_bal} &
\partial_t (\frac{1}{2} h \vert \bfu \vert^2) + \dive (\frac{1}{2} h  \vert \bfu \vert^2 \bfu) + \bfu \cdot \nabla p  + gh \bfu \cdot \nabla z = 0.
\end{align}
Summing these equations, we obtain en entropy equality of the form $\partial_t \eta + \dive \Phi = 0$, where the entropy-entropy flux pair $(\eta, \Phi)$ is given by:
\begin{equation}\label{def:entropy}
\eta = \frac 1 2  h|\bfu|^2 + \frac 1 2 g h^2 + ghz \text{ and } \Phi =(\eta +  \frac 1 2 g h^2) \bfu.
\end{equation}
For non regular functions the above manipulations are no longer valid, and the entropy inequality $\partial_t \eta + \dive \Phi \leq 0$ is satisfied in a distributional sense.

\medskip
In this paper, we build a decoupled scheme, involving only explicit steps; the resulting approximate solutions are shown to satisfy some discrete equivalent of \eqref{pot_bal} and \eqref{kin_bal}; furthermore, under some convergence and boundedness assumptions, the approximate solutions are shown in Section \ref{sec:cons} to converge to a weak solution of \eqref{eq:sw} and to satisfy a weak entropy inequality. 
%
%
\section{Mesh and space discretizations}\label{sec:discop}

Let $\Omega$ be a connected subset of $\xR^2$ consisting in a union of rectangles whose edges are assumed to be orthogonal to the canonical basis vectors, denoted by $(\bfe^{(1)}, \bfe^{(2)})$.

\begin{defn}[MAC grid] \label{def:MACgrid}
A discretization $(\mesh, \edges)$ of $\Omega$ with a staggered rectangular grid (or MAC grid), is defined by:
\begin{list}{--}{\itemsep=0.ex \topsep=0.5ex \leftmargin=1.cm \labelwidth=0.7cm \labelsep=0.3cm \itemindent=0.cm}
\item A primal grid $\mesh$ which consists in a conforming structured partition of $\Omega$ in rectangles, possibly non uniform.
A generic cell of this grid is denoted by $K$, and its mass center by $\bfx_K$.
The scalar unknowns (water height and pressure) are associated to this mesh.
\item The set of all edges of the mesh $\edges$, with $\edges= \edgesint \cup \edgesext$, where $\edgesint$ (resp. $\edgesext$) are the edges of $\edges$ that lie in the interior (resp. on the boundary) of the domain.
The set of edges that are orthogonal to $\bfe\ei$ is denoted by $\edgesi$, for $\iinud$.
We then have $\edgesi= \edgesinti \cup \edgesexti$, where $\edgesinti$ (resp. $\edgesexti$) are the edges of $\edgesi$ that lie in the interior (resp. on the boundary) of the domain.

For $\edge\in\edgesint$, we write $\edge = K|L$ if $\edge = \partial K \cap \partial L$.
A dual cell $D_\edge$ associated to an edge $\edge \in\edges$ is defined as follows:
\begin{list}{-}{\itemsep=0.ex \topsep=0.ex \leftmargin=1.cm \labelwidth=0.7cm \labelsep=0.1cm \itemindent=0.cm}
\item if $\edge=K|L \in \edgesint$ then $D_\edge = D_{K,\edge}\cup D_{L,\edge}$, where $D_{K,\edge}$ (resp. $D_{L,\edge}$) is the half-part of $K$ (resp. $L$) adjacent to $\edge$ (see Fig. \ref{fig:mesh}); 
\item if $\edge \in \edgesext$ is adjacent to the cell $K$, then $D_\edge=D_{K,\edge}$.
\end{list}
For each dimension $i=1,2$, the domain $\Omega$ is partitioned in dual cells: $\Omega = \cup_{\edge \in \edgesi} D_\edge$, $\iinud$; the $i^{th}$ partition is refered to as the $i^{th}$ dual mesh; it is associated to the $i^{th}$ velocity component, in a sense which is clarified below.
The set of the edges of the $i^{th}$ dual mesh is denoted by $\edgesdi$ (note that these edges may be orthogonal to any vector of the basis of $\xR^2$ and not only $\bfe\ei$) and is decomposed into the internal and boundary edges: $\edgesdi = \edgesdinti\cup \edgesdexti$.
The dual edge separating two duals cells $D_\edge$ and $D_{\edge'}$ is denoted by $\edged=\edge|\edge'$. We denote by $D_\edged$ the dual cell associated to a dual edge $\edged \in \edgesd$ defined as follows:
\begin{list}{-}{\itemsep=0.ex \topsep=0.ex \leftmargin=1.cm \labelwidth=0.7cm \labelsep=0.1cm \itemindent=0.cm}
\item if $\edged =\edge|\edge' \in \edgesd_{int}$ then $D_\edged = D_{\edge,\edged}\cup D_{\edge',\edged}$, where $D_{\edge,\edged}$ (resp. $D_{\edge',\edged}$) is the half-part of $D_\edge$ (resp. $D_{\edge'}$) adjacent to $\edged$ (see Fig. \ref{fig:mesh}); 
\item if $\edged \in \edgesd_{ext}$ is adjacent to the cell $D_\edge$, then $D_\edged=D_{\edge,\edged}$.
\end{list}
\end{list}
\end{defn}

\medskip 
In order to define the scheme, we need some additional notations.
The set of edges of a primal cell $K$ and of a dual cell $D_\edge$ are denoted by $\edgesK$ and $\edgesd(D_\edge)$ respectively.
For $\edge \in \edges$, we denote by $\bfx_\edge$ the mass center of $\edge$.
The vector $\bfn_{K,\edge}$ stands for the unit normal vector to $\edge$ outward $K$. 
In some cases, we need to specify the orientation of various geometrical entities with respect to the axis:
\begin{list}{-}{\itemsep=0.ex \topsep=0.5ex \leftmargin=1.cm \labelwidth=0.7cm \labelsep=0.3cm \itemindent=0.cm}
\item a primal cell $K$ will be denoted $K = [\overrightarrow{\edge \edge'}]$ if $\edge, \edge' \in \edgesi(K)$ for some $\iinud$ are such that $(\bfx_{\edge'} - \bfx_\edge) \cdot \bfe\ei >0$;
\item we write $\edge =\overrightarrow{K|L}$ if $\edge \in\edgesi$, $\edge=K|L$ and $\overrightarrow{\bfx_K\bfx_L}\cdot \bfe\ei>0$ for some $\iinud$;
\item the dual edge $\edged$ separating $D_\edge$ and $D_{\edge'}$ is written $\edged = \overrightarrow{\edge|\edge'}$ if $\overrightarrow{\bfx_\edge \bfx_{\edge'}}\cdot \bfe\ei>0$ for some $\iinud$.
\end{list}

\begin{figure}[tp]\centering
\includegraphics[scale=.7]{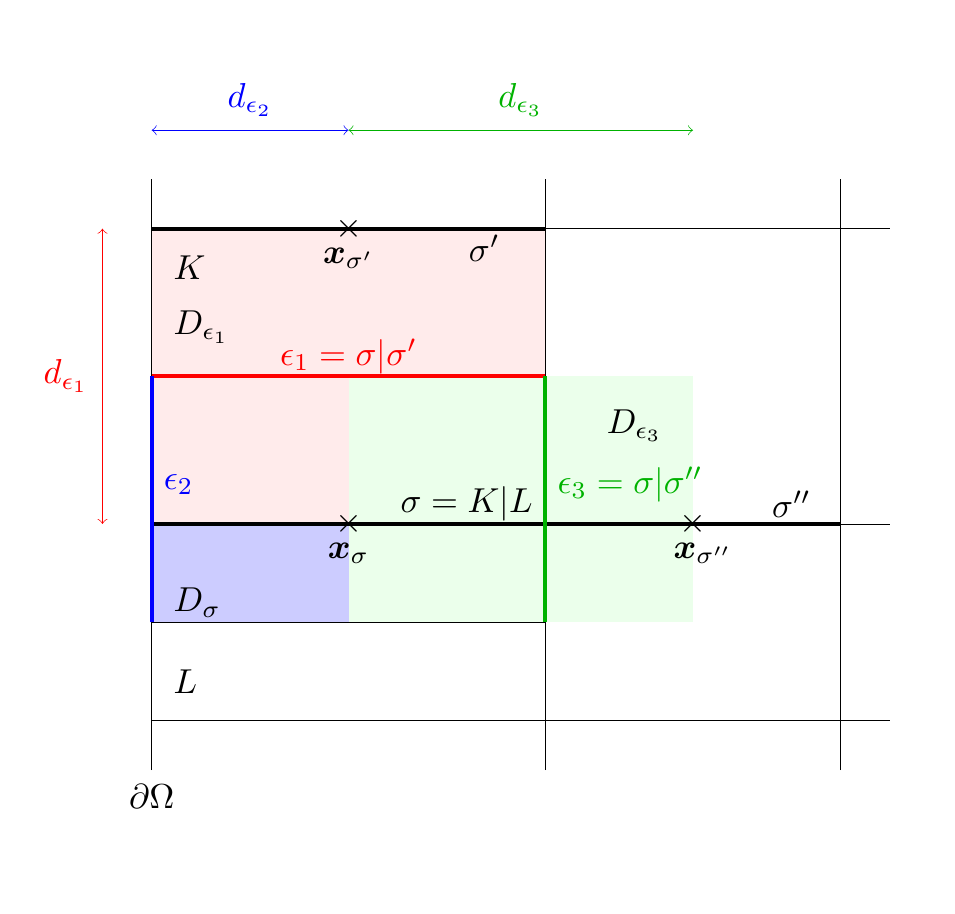}
\caption{Notations for control volumes and dual cells (in two space dimensions, for the second component of the velocity).}
\label{fig:mesh}
\end{figure}

The size $\delta_\mesh$ of the mesh and its regularity $\eta_\mesh$ are defined by:
\begin{equation}  \label{regmesh} 
\delta_\mesh=\max_{K\in\mesh }  \diam(K), \mbox{ and }
\eta_\mesh = \max \Bigl\{ \frac{|\edge|}{|\edge'|},\ \edge \in \edgesi,\ \edge' \in \edgesj,\ i, j = 1, 2,\ i\not= j \Bigr\},
\end{equation}
where $|\cdot|$ stands for the one (or two) dimensional measure of a subset of $\xR$ (or $\xR^2$).

\medskip
The discrete velocity unknowns are associated to the dual cells and are denoted by $(u_{i,\edge})_{\edge\in\edgesi}$, $\iinud$, while the scalar unknowns (discrete water height and pressure) are associated to the primal cells and are denoted respectively by $(h_K)_{K\in\mesh}$ and $(p_K)_{K\in\mesh}$.
The scalar unknown space $L_\mesh$ is defined as the set of piecewise constant functions over each grid cell $K$ of $\mesh$, and the discrete $i^{th}$ velocity space $\Hmeshi$ as the set of piecewise constant functions over each of the grid cells $D_\edge,\ \edge\in\edgesi$.
As in the continuous case, the Dirichlet boundary conditions are taken into account by defining the subspaces $\Hmeshizero \subset \Hmeshi,\ \iinud$ as follows
\[
\Hmeshizero=\Bigl\{u_i\in\Hmeshi,\ u_i(\bfx) = 0,\, \forall \bfx\in D_\edge,\ \edge \in \edgesexti \Bigr\}.
\]
We then set $\Hmeshzero= \Hmeshunzero \times \Hmeshdezero$.
Defining the characteristic function $\characteristic_A$ of any subset $A \subset \Omega$ by $\characteristic_A(\bfx)=1$ if $\bfx \in A$ and $\characteristic_A(\bfx)=0$ otherwise, the functions $\bfu =(u_1,  u_2) \in \Hmeshzero$, may then be written:
\begin{equation} \label{eq:ufx}
\ui(\bfx) = \sum_{\edge\in \edgesi} u_{i,\edge} \characteristic_{D_\edge}(\bfx),\ \iinud.
\end{equation}
For $\bfu \in \Hmeshzero$, let $\llbracket u_i \rrbracket_{\edged} = |u_{i,\edge} - u_{i,\edge'}|$, for $\edged = \edge|\edge' \in \edgesd\ei_{int}, \iinud$.
In the same way the functions $h \in L_\mesh$ are defined by $h(\bfx) = \sum_{K \in \mesh} h_K \characteristic_K(\bfx)$ and the notation $\llbracket \ \rrbracket_\edge$ refers to $\llbracket h \rrbracket_{\edge} = |h_K-h_L|, \, \text{ for }  \edge = K|L \in \edgesint(K)$.
%
%
\section{A decoupled explicit scheme} \label{sec:time} \label{sch}

\paragraph{Description of the scheme}
Let us consider a uniform discretisation $0 = t_0 < t_1 < \cdots < t_N = T$ of the time interval $(0,T)$, and let $\delta t = t_{n+1} - t_n$ for $n = 0, 1, \cdots, N - 1$ be the (constant) time step. 
The discrete velocity $\bfu$ and water height $h$ unknowns are defined by: 
\begin{align*} &
\bfu(\bfx, t) = \sum_{n = 0}^{N-1} \bfu^{n+1}(\bfx) \characteristic_{[t_n, t_{n+1})}(t),\,  \text{ with } \bfu^{n+1} \in \Hmeshzero,
\\[-1ex] &
h(\bfx, t) = \sum_{n = 0}^{N-1} h^{n+1}(\bfx) \characteristic_{[t_n, t_{n+1})}(t), \text{ with } h^{n+1} \in L_{\mesh},
\end{align*}
where $ \characteristic_{[t_n, t_{n+1})}$ is the characteristic function of the interval $[t_n , t_{n+1} )$ and the space functions $\bfu^n$ and $h^n$ take the form defined in the previous section.
We propose the following decoupled discretisation of the system (\ref{eq:sw}), written in compact form, with the various discrete operators defined below.
\begin{subequations}\label{scheme}
\begin{align}  &
\mbox{{\bf Initialisation}:} \qquad \label{sch-init}  
\bfu^0= \mathcal{P}_\edges \bfu_0, \, h^{0} = \mathcal{P}_\mesh h_0, \,p^{0} = \frac 1 2 g (h^{0})^{2}.
\\[1ex] 
\nonumber & \mbox{{\bf Iteration}} \, n, \, 0 \leq n \leq N-1: \text{solve for } \bfu^{n+1} \in \Hmeshzero, h^{n+1} \in L_{\mesh} \mbox{ and } p^{n+1} \in L_{\mesh}: 
\\[0.5ex] \label{sch-mass} & \hspace{15ex}
\eth_t h^{n+1} + \dive_\mesh\,(h^{n} \bfu^{n}) =0, 
\\[0.5ex] \label{sch-pressure}
& \hspace{15ex}
p^{n+1} = \frac{1}{2} g (h^{n+1})^2,
\\[0.5ex] \label{sch-vitesse} & \hspace{15ex}
\eth_t (h \bfu)^{n+1} + {\bfC}_\edges (h^{n} \bfu^{n})\bfu^{n} + \nabla_\edges p^{n+1} + g\, {\bfI}_\edges h^{n+1} \ \nabla_\edges z = 0,
\end{align}
\end{subequations}

\medskip
{\it Projection operators} - The operators $\mathcal{P}_\edges$ and $\mathcal{P}_\mesh$ used in the initialisation step are defined by $\mathcal{P}_\edges = (\mathcal{P}_{\edges^{(i)}})_{i=1,\cdots,d}$ with\\
\begin{equation} \label{def:proj:oper1} 
\begin{array}{l| l} \displaystyle
\mathcal{P}_{\edges^{(i)}}: \quad 
& \quad
L^1(\Omega) \longrightarrow \Hmeshizero 
\\[1ex] & \displaystyle \quad
v \;\longmapsto \displaystyle \mathcal{P}_\edgesi v = \sum_{\edge \in \edgesinti} v_{\edge}\, \characteristic_{D_\edge}
\mbox{ with }
v_{\edge} = \frac 1 {|D_\edge|}\int_{D_\edge} v(\bfx) \dx, \mbox{ for } \edge \in \edgesinti.
\end{array}
\end{equation} 
For $q \in L^2(\Omega)$, $\mathcal{P}_\mesh q \in L_\mesh$ is defined by: 
\begin{equation}\label{def:proj:oper2}
\displaystyle \mathcal{P}_\mesh q = \sum_{K \in \mesh} q_K \characteristic_K \text{ with } \displaystyle q_K = \dfrac{1}{|K|} \int_{K} q(\bfx) \dx \text{ for } K \in \mesh.
\end{equation}

{\it Discrete time derivative} - The symbol $\eth_t$  denotes the discrete time derivative for both  water height and momentum:
\begin{align*}
& 
 \eth_t   h^{n+1} = \sum_{K\in\mesh}  \frac 1 {\delta t}(h^{n+1}_K - h^{n}_K) \characteristic_K, \qquad \eth_t (h \bfu)^{n+1} = (\eth_t  (h u_1)^{n+1}, \ldots,\eth_t  (h u_d)^{n+1}) \\
& \text{with } \,\eth_t(h u_i)^{n+1}=\! \sum_{\edge\in {\edges^{(i)}}}  \frac 1 {\delta t}(h^{n+1}_{D_\edge} u^{n+1}_{i, \edge} - h^{n}_{D_\edge} u^{n}_{i,\edge}) \characteristic_{D_\edge}, \, \iinud,
\end{align*}
where $h_{D_\edge}$ is the discrete water height in the dual cell, which is computed from the primal unknowns $(h_K^n)_{\nnn, K \in \mesh}$ and defined so as to satisfy a discrete mass balance, see below.
\\

\noindent 
{\it Discrete divergence and gradient operators} -
The discrete divergence operator $\dive_\mesh$ is defined by:
\begin{align} \label{eq:div} &
\begin{array}{l| l} \displaystyle
\dive_\mesh:
 &  
\Hmeshzero \longrightarrow L_{\mesh,0}
\\[1ex] & \displaystyle \quad
\bfu \longmapsto \dive_\mesh\, (h \bfu) = \sum_{K\in\mesh} \dive_K(h \bfu) \characteristic_K, \text{ with } \dive_K(h \bfu) = \frac 1 {|K|} \sum_{\edge\in\edges(K)} F_{K,\edge}, 
\end{array}
\end{align}
where $F_{K,\edge}$ is the (conservative) numerical mass flux, defined by $F_{K,\edge} = |\edge|\ h_{\edge} u_{K,\edge}$ with   $u_{K,\edge} =  u_{i,\edge} \nKedge \cdot \bfe\ei  \mbox{ for } \edge \in \edgesint^{(i)},\ \iinud,$ while $h_{\edge}$ is approximated by the first order upwind scheme namely,  for $\edge = K|L \in \edges_{int}$, $h_{\edge} = h_K$ if $u_{K,\edge} \geq 0$ and $h_{\edge} = h_L$ otherwise.

\medskip
The discrete gradient operator applies to the pressure and the topography and is defined by:
\begin{equation*}\label{eq:grad}
\begin{array}{l|l}
\nabla_\edges:\quad
& \quad
L_\mesh \longrightarrow \Hmeshzero 
\\[1ex] & \displaystyle \quad
p \longmapsto \nabla_\edges p, 
\end{array}
\end{equation*}
with $\text{ for } \iinud$: 
\begin{equation} 
(\nabla_\edges p)_i  =   \sum_{\edge \in \edges_{int}^{(i)}} (\eth_i p)_{\sigma}  \characteristic_{D_\edge} \,\, \text{with} \text{ for } \edge = \overrightarrow{K|L},\, (\eth_i p)_{\sigma}= \frac{|\edge|}{|D_\edge|}\ (p_L - p_K).
\end{equation}
The above defined discrete divergence and gradient operators satisfy the following div-grad duality relationship \cite[Lemma 2.5]{gal-18-conv}:
\begin{equation*}
\text{for } p \in L_\mesh, \, \bfu \in\Hmeshzero, \quad  \int_\Omega p \,\dive_\mesh(\bfu) \dx + \int_\Omega \nabla_\edges p \cdot (\bfu) \dx = 0.
\end{equation*}

\medskip
{\it Discrete convection operator} -- The discrete nonlinear convection operator $\bfC_\edges(h \bfu) $ is linked to the discrete divergence operator on the dual mesh by the relation ${\bfC}_\edges (h \bfu)\bfu = \dive_{\edges}(h \bfu \otimes \bfu)$, where the full discrete convection operator $\bfC_\edges(h \bfu)$ is defined by: 
\[
\bfC_\edges (h \bfu)\, \bfu = \bigl(C_{\edges^{(1)}}(h \bfu)\, u_1,  C_{\edges^{(2)}}(h \bfu)\, u_2 \bigr),
\]
and the $i$-th component $C_\edgesi(h\bfu)$ of the convection operator is defined by:
\begin{equation}\label{eq:conv}
\begin{array}{l|l}
C_\edgesi(h \bfu): \quad
& \quad
\Hmeshizero \longrightarrow \Hmeshizero
\\ & \displaystyle \quad 
u_i \longmapsto C_\edgesi(h \bfu)\, u_i = \sum_{\edge \in \edgesinti} \dive_{\edges\ei}(h u_i \bfu)\; \characteristic_{D_\edge},
\\ & \displaystyle
\text{ with } \dive_{\edges\ei}(h u_i \bfu) = \frac 1{|D_\edge|} \sum_{\edged \in \edgesd^{(i)}(D_\edge)} F_{\edge,\edged} u_{i,\edged} ,
\end{array}
\end{equation}
where $u_{i,\edged}$ is approximated by the upwind technique with respect to the sign of $F_{\edge,\edged}$.
The quantity $ F_{\edge,\edged}$ is the numerical mass flux through $\edged$ outward $D_\edge$; it must be chosen carefully to ensure some stability properties of the scheme as in \cite{gal-18-conv, her-18-con}. 
Indeed we recall that in order to derive a discrete kinetic energy balance (Lemma \ref{lem:disc_kin} below), it is necessary that a discrete equation of the mass balance holds  in the dual mesh, namely:
\begin{equation}\label{mass_dual}
\dfrac{|D_\edge|}{\delta t}(h_{D_\edge}^{n+1} - h_{D_\edge}^{n}) +  \dive_{\edges}(h^n \bfu^n) = 0,\quad
\text{with}\quad |D_\edge|\ \dive_{\edges}(h^n \bfu^n) = \sum_{\edged \in \edgesd(D_\edge)} F^n_{\edge,\edged}.
\end{equation}
The water height $h_{D_\edge}$ and the flux $F_{\edge,\edged}$ are computed from the primal unknowns and fluxes so as to satisfy this latter relation thanks to the discrete mass balance on the primal mesh \eqref{sch-mass}.
For $\edge = K|L \in \edgesint$, the water height $h_{D_\edge}$ is defined as a weighted average between $h_K$ and $h_L$:
\begin{equation}\label{discete_water_dual}
|D_\edge|\ h_{D_\edge} = |D_{K,\edge}|\ h_K + |D_{L,\edge}|\ h_L,
\end{equation}
where $D_\edge$, $D_{K,\edge}$ and $D_{L, \edge}$ are defined in Definition \ref{def:MACgrid}.
\begin{figure}[tp] 
\centering
\includegraphics[scale=.7]{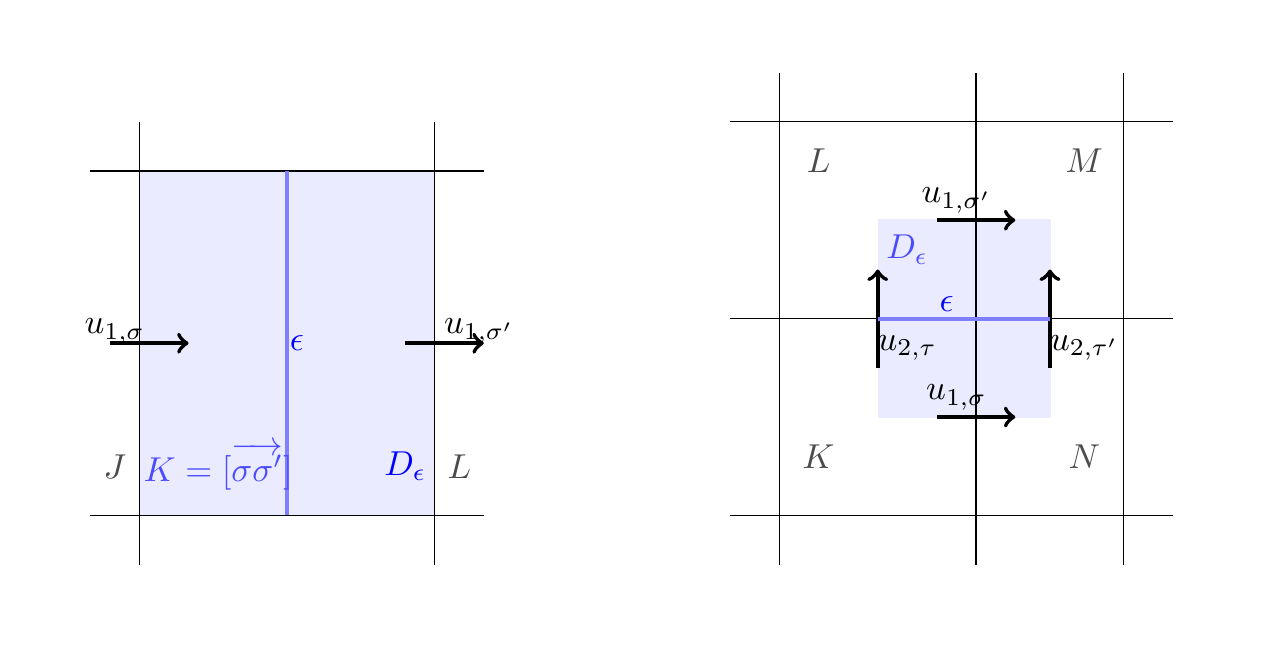}
\caption{Notations for the definition of the momentum flux on the dual mesh for the first component of the velocity- left: first case - right: second case.}
\label{fig:convection}
\end{figure}
The numerical flux  $F_{\edge,\edged}$ on the internal dual edges, is defined according to the location of the edges as follows:
\begin{list}{-}{\itemsep=0.ex \topsep=0.5ex \leftmargin=1.cm \labelwidth=0.7cm \labelsep=0.3cm \itemindent=0.cm}
\item First case -- The vector $\bfe\ei$ is normal to $\edged$, and $\edged$ is included in a primal cell $K$,  with $K = [\overrightarrow{\edge \edge'}]$ (see Definition \ref{def:MACgrid} and Figure \ref{fig:convection} on the left for $i=1$).
Then for a dual edge $\edged \in \edgesd\ei$ such that $\edged=\overrightarrow{\edge|\edge'}$, the flux $F_{\edge,\edged}$ through the edge $\edged$ is given by:
\begin{equation}\label{eq:flux_eK}
F_{\edge,\edged}= \frac 1 2  (F_{K,\edge'} - F_{K,\edge}) = \frac 1 2 |\edged|\ (h_{\edge}u_{i,\edge}  + h_{\edge'}u_{i,\edge'}),
\end{equation}
since $|\edge| = |\edge '| =|\edged| $. 

\smallskip
\item Second case -- The vector $\bfe\ei$ is tangent to $\edged$, and $\edged$ is the union of the halves of two primal edges $\edgeperp$ and $\edgeperp'$ such that $\edgeperp=\overrightarrow{K|L}$, $\edgeperp\in \edges(K)$ and $\edgeperp' = \overrightarrow{N|M} \in \edges(N)$ (see Definition \ref{def:MACgrid} and Figure \ref{fig:convection} on the right for $i=2$).
The flux numerical through $\edged$ is then given by:
\begin{equation}\label{eq:flux_eorth}
F_{\edge,\edged} = \frac 1 2\ (F_{K\edgeperp} + F_{L\edgeperp'}) = \frac 1 2\ (|\edgeperp|\ h_{\edgeperp} u_{,\edgeperp} + |\edgeperp'|\ h_{\edgeperp'}u_{i,\edgeperp'}).
\end{equation}
\end{list}
Note that the numerical momentum flux on a dual edge is conservative.
It is easy to check that the unknowns $h_{D_\edge}^{n}$ and $F^n_{\edge,\edged}$ thus defined satisfy the discrete dual mass balance \eqref{mass_dual}.

\medskip  
{\it Discrete water height on the dual mesh, for the topography term} -- In equation \eqref{sch-vitesse} the interpolation operator ${\bfI}_\edges$ is defined as the mean value of the water height:
\begin{equation} \label{hedge}
{\bfI}_\edges h = \sum_{\edge \in \edges_{int}} h_{\edge,c} \characteristic_{D_\edge} \mbox{ with } h_{\edge,c} =
\begin{cases} \frac{1}{2}(h_K+h_L) \text{ for } \edge = K|L \in \edges_{int},\\[1ex] h_K \text{ for }  \edge \in \edges_{ext} \cap \edges(K).\end{cases}
\end{equation}
This choice is important to preserve steady states, see Lemma \ref{lem:lake}.
%
%
\section{Properties of the scheme}

The scheme \eqref{scheme} enjoys some interesting properties, which we now state.
First of all, thanks to the upwind choice for $h^n$ in \eqref{eq:mass}, the positivity of the water height is preserved under a CFL like condition.

\medskip
\begin{lemma}[Positivity of the water height]\label{sha:lem:positivity}
Let $\nnud$, let $(h_K^n,\ u_{i,\edge}^n)_{K \in \mesh, \ \edge \in \edges^{(i)}}$  be given  and such that $h_K^n \ge 0$, for all  $K \in \mesh$, and let  $h_K^{n+1}$ be computed by \eqref{sch-mass}. 
Then $h_K^{n+1} \geq 0$, for all $K \in \mesh$ under the following CFL condition, 
\begin{equation}\label{cfl:posit:euler}
\delta t \leq \frac{|K|}{\displaystyle \sum_{\edge \in \edges(K)} |\edge|\,|\bfu^n_{K,\edge}|}.
\end{equation}
\end{lemma}

\medskip
Second, thanks to the choice \eqref{hedge} for the reconstruction of the water height, the "lake at rest" steady state is preserved by the scheme. 

\begin{lemma}[Steady state "lake at rest"]\label{lem:lake}
Let $\nnud$, $C \in \xR_+$; let $\bfu^{n+1} \in \Hmeshzero$ and $h^{n+1}\in L_\mesh$ be a solution to \eqref{sch-mass}-\eqref{sch-vitesse} with $\bfu^n=0$ and $h^n + z = C$, where $C$ is a given real number.
Then $\bfu^{n+1}=0$ and $h^{n+1}+z=C$.   
\end{lemma}

\medskip
As a consequence of the careful discretisation of the convection term, the scheme satisfies a discrete kinetic energy balance, as stated in the following lemma. 
The proof of this result is an easy adaptation of \cite[Lemma 3.2]{her-13-exp}.

\begin{lemma}[Discrete kinetic balance]\label{lem:disc_kin}
A solution to the scheme \eqref{scheme} satisfies the following equality, for $i=1,2$, $\edge \in \edges^{(i)}$ and $0 \leq n \leq N-1$:
\begin{multline}\label{disc-kinet}
\dfrac 1 {2\,\delta t} (h_{D_\edge}^{n+1} (u^{n+1}_{i,\edge})^2 - h_{D_\edge}^n (u^n_{i,\edge})^2)
+ \dfrac{1}{2\ |D_\edge|} \sum_{\edged \in \edgesd^{(i)}(D_\edge)} F^n_{\edge,\edged} (u^n_{i,\edged})^2
\\
+ u_{i,\edge}^{n+1} (\eth_i p^{n+1})_\edge 
+ g\, h_{\edge,c}^{n+1} \, u_{i,\edge}^{n+1} (\eth_i z)_\edge = - R_{i,\edge}^{n+1},
\end{multline}
with $R_{i, \edge}^{n+1} \geq 0$ under the CFL like restriction:
\begin{equation}\label{cfl:positif:energy}
\forall \edge \in \edges^{(i)}, \qquad \delta t \leq \dfrac{|D_\edge| \, h^{n+1}_{D_\edge}}{ \displaystyle \sum_{\edged \in \edgesd(D_\edge)} (F_{\edge,\edged}^n)^{-} }.
\end{equation} 
\end{lemma}

The scheme also satisfies the following potential energy balance \cite[Lemma 3.3]{her-13-exp}.

\begin{lemma}[Discrete elastic potential balance]
Let, for $K \in \mesh$ and $0 \leq n \leq N$ the potential energy be defined by $(E_p)_K^n=\frac 1 2 g\,(h_K^n)^2$.
A solution to the scheme \eqref{scheme} satisfies the following equality, for $K \in \mesh$ and $0 \leq n \leq N-1$:
\begin{equation}\label{disc-pot}
\eth_t E_p^{n+1} + \dive_K(E_p^n \bfu^n ) + p_K^n \dive_K (\bfu^n) = - R_K^{n+1},
\end{equation}
with
\begin{equation}\label{residu:pot:bound}
R_K^{n+1} \geq \dfrac{1}{|K|} g
\sum_{\edge \in \edges(K)} |\edge|\ u_{K,\edge}^n h^n_{\edge} (h^{n+1}_{K} - h_K^n).
\end{equation}
\end{lemma}

Note that the right-hand side of Equation \eqref{residu:pot:bound} may be negative, and thus the quantities $R_K^{n+1}$ also.
This is specific to explicit schemes (for implicit or pressure-correction schemes  \cite{her-14-imp}, this residual is non-negative) and prevents  getting a stability estimate for the scheme.
However, combining the two previous lemmas allows to prove that convergent sequences of solutions to the scheme satisfy an entropy inequality, as depicted in the next section.
To this purpose, we will pass to the limit in a discrete entropy balance which is built as follows.
Let $K \in \mesh$ and let us denote by $(E_k)_K^n$ the following quantity, which may be seen as a kinetic energy associated to $K$:
\[
(E_k)_K^n= \frac 1 {4\ |K|} \sum_{i=1}^2 \sum_{\edge \in \edges(K)\cap \edges^{(i)}} |D_\edge|\ h_{D_\edge}^n (u^n_{i,\edge})^2.
\]
Then, for $\edge_0 \in \edges(K)$, we define a kinetic energy flux, which we denote by $G_{K,\edge_0}^n$, as follows.
Let us suppose, for instance, that $\edge_0 \in \edges^{(1)}$.
We denote by $\edged$ the face of $D_{\edge_0}$ parallel to $\edge_0$ and included in $K$ and by $\edged'$ the opposite face of $D_{\edge_0}$.
In addition, $\edge_0$ is the union of two half-faces of the dual mesh associated to the second component of the velocity, which we denote by $\tau$ and $\tau'$, and we denote by $\edge$ and $\edge'$ the two faces of $K$ belonging to $\edges^{(2)}$ such that $\tau \in \edgesd(D_\edge)$ and $\tau' \in \edgesd(D_{\edge'})$.
We then have:
\[
G_{K,\edge_0}^n = \frac 1 4 \Bigl[-F^n_{\edge_0,\edged}\, (u^n_{1,\edged})^2 + F^n_{\edge_0,\edged'}\, (u^n_{1,\edged'})^2
+ F^n_{\edge,\tau}\, (u^n_{2,\tau})^2 + F^n_{\edge,\tau'}\, (u^n_{2,\tau'})^2 \Bigr] 
\]
Multiplying the kinetic energy balance equation \eqref{disc-kinet} associated to each face $\edge$ of $K$ by $\frac 1 2\ |D_\edge|$ and summing the four obtained relations with \eqref{disc-pot}, we get
\begin{multline*}
\frac{|K|}{\delta t}\ \bigl[(E_k)_K^{n+1} +(E_p)_K^{n+1} - (E_k)_K^n-(E_p)_K^n\bigr]
+ \sum_{\edge \in \edges(K)} \bigl[ G_{K,\edge}^n + F_{K,\edge}^n\ (E_p)_\edge^n \bigr]
\\
+ \sum_{\edge \in \edges(K),\ \edge=K|L} |\edge|\ \frac 1 2\ (p^{n+1}_L-p^{n+1}_K)\ u_{K,\edge}^{n+1}
+ \sum_{\edge \in \edges(K),\ \edge=K|L} |\edge|\ \frac 1 4 g\ (h_K^{n+1}+h_L^{n+1})\ (z_L-z_K)\,u_{K,\edge}^{n+1}
= - T_K^{n+1},
\end{multline*}
where $T_K^{n+1}$ collects the residual terms in \eqref{disc-kinet} and \eqref{disc-pot}, and thus $T_K^{n+1} \geq R_K^{n+1}$.
We now remark that, thanks to the discrete mass balance equation and the fact that the topography does not depend on time,
\[
\frac 1 2 \sum_{\edge \in \edges(K)} F_{K,\edge}^n (z_L-z_K) =
\frac{|K|}{\delta t}\ (h_K^{n+1} z_K - h_K^n z_K) + \frac 1 2 \sum_{\edge \in \edges(K),\ \edge=K|L} F_{K,\edge}^n\ (z_K+z_L),
\]
and we finally obtain the following discrete entropy balance:
\begin{multline}\label{eq:ent_disc}
\frac{|K|}{\delta t}\ \bigl[(E_k)_K^{n+1} +(E_p)_K^{n+1} + g\,h_K^{n+1}\,z_K - (E_k)_K^n-(E_p)_K^{n} - g\,h_K^{n}\,z_K\bigr]
\\
+ \sum_{\edge \in \edges(K)} \bigl[ G_{K,\edge}^n + F_{K,\edge}^n\ (E_p)_\edge^n + \frac 1 2\, F_{K,\edge}^n (z_K+z_L) \bigr]
\\ + \sum_{\edge \in \edges(K),\ \edge=K|L} |\edge|\ \frac 1 2\ (p^{n+1}_K+p^{n+1}_L)\ u_{K,\edge}^{n+1}
= -(R_e)_K^{n+1},
\end{multline}
with
\begin{multline}
(R_e)_K^{n+1} \geq  T_K^{n+1} + g \sum_{\edge \in \edges(K)} \bigl[ \frac 1 2 F_{K,\edge}^n -\frac 1 4 |\edge|\ (h^{n+1}_K+h^{n+1}_L)\ u_{K,\edge}^{n+1} \bigr]\ (z_L-z_K) \\ + \sum_{\edge \in \edges(K),\ \edge=K|L} |\edge|\ \frac 1 2\ (p^{n+1}_K u_{K,\edge}^{n+1}-p^n_K u_{K,\edge}^n) .
\end{multline}
%
%
\section{Consistency analysis}\label{sec:cons}

The objective of this section is to show that the schemes are consistent in the Lax-Wendroff sense, namely that if a sequence of solutions is controlled in suitable norms and converges to a limit, this latter necessarily satisfies a weak formulation of the continuous problem. 

\medskip
A weak solution to the continuous problem satisfies, for any $\varphi \in C^\infty_c \bigl(\Omega \times [0,T)\bigr)$ ($\bfvarphi \in C^\infty_c \bigl(\Omega \times [0,T)\bigr)^2$): 
\begin{subequations} \begin{align}\label{sha:eq:pbw_mass} &
\int_0^T \int_\Omega \Bigl[ h \, \partial_t \varphi + h\, u \, \cdot \nabla \varphi \Bigr]\dx \dt
+\int_\Omega h_0(\bfx) \, \varphi(\bfx,0) \dx= 0,
\\ \label{sha:eq:pbw_mom} &
-\int_0^T \int_\Omega \Bigl[ h\, \bfu \cdot \partial_t \bfvarphi + (h \bfu \otimes \bfu):\bfvarphi + \frac 1 2\ g \, h^{2} \dive(\bfvarphi)  + g \, h \, \nabla(z) \bfvarphi \Bigr] \dx \dt 
\\ &\hspace{6cm} \nonumber 
-\int_\Omega h_0(\bfx) \, \bfu_0(\bfx)\cdot \varphi(\bfx,0) \dx = 0. 
\end{align}\label{sha:eq:pbw} 
\end{subequations}
This system is supplemented with a weak entropy inequality, for any nonnegative test functions $\varphi \in C^\infty_c \bigl(\Omega \times [0,T), \mathbb R_+\bigr)$ : 
\begin{equation}\label{eq:weakentropy}
-\int_0^T \int_\Omega \Bigl[\eta \, \partial_t \varphi + \bf{\Phi} \cdot \nabla \varphi \Bigr] \dx \dt 
-\int_\Omega \eta_0(\bfx)\, \varphi(\bfx,0) \dx \leq 0,
\end{equation}
with $\eta$ and $\bf{\Phi}$ defined by \eqref{def:entropy}.

\medskip
Before stating the global weak consistency of the scheme \eqref{scheme}, some definitions and  estimate assumptions are needed. 

\medskip
Let $(\mesh^{(m)},\edges^{(m)})_\mnn$ be a sequence of meshes in the sense of Definition \ref{def:MACgrid} and let ($h^{(m)}$ $\bfu^{(m)})_\mnn$  be the associated sequence of solutions of the scheme (\ref{scheme})).

\medskip
\textbf{Assumed estimates }- We need also some a priori estimates on the sequence of discrete solutions $\displaystyle (h\m, \ \bfu\m)_{m \in \xN}$ in order to prove the consistency result we are seeking.
First of all we assume that $h\m > 0$, $\forall m \in \xN$ which can be obtained under the CFL condition \eqref{cfl:posit:euler}.
Furthermore:
\begin{list}{--}{\itemsep=0.ex \topsep=0.5ex \leftmargin=1.cm \labelwidth=0.7cm \labelsep=0.3cm \itemindent=0.cm}
\item  The water height $h\m$ and its inverse are uniformly bounded in $\xL^\infty(\Omega \times (0,T))$, \ie there exists some constants $C, C' \in \xR_+^\ast$ such that for $m \in \xN$ and $0 \leq n < N\m$:
\begin{equation}\label{infini:boud:h}
1/C < (h\m)^n_K\leq C, \quad 1/C' < 1/( h\m)^n_K\leq C' \quad \forall K \in \mesh\m
\end{equation} 
\item The velocity $\bfu\m$ is also uniformly bounded in $\xL^\infty(\Omega \times (0,T))^2$: 
\begin{equation}\label{infini:bound:u}
|(\bfu\m)^n_\edge| \leq C, \quad \forall \edge \in \edges\m.
\end{equation}
\end{list}
Finally, the weak consistency to the entropy inequality is only proved under additional assumptions.
First we need the following condition on the space and time steps, which is stronger than a CFL condition:
\begin{equation}\label{eq:ent1}
\frac {\delta t^{(m)}}{\delta_{\mesh^{(m)}}} \to 0 \mbox{ as } m\to +\infty
\end{equation}
Second, the $L^1(\Omega,BV)$ norm of the height is required to be bounded, \ie\ there exists one constant $C$ such that,  for $m \in \xN$,
\begin{equation}\label{eq:ent2}
\sum_{n=0}^{N-1} \sum_{K\in\mesh} |K|\ |(h\m)_K^{n+1} - (h\m)_K^n| \leq C.
\end{equation}

We are now in position to state the following consistency result.

\begin{thm}[Weak consistency of the scheme]\label{Euler:weak:const}
Let $(\mesh^{(m)},\edges^{(m)})_\mnn$ be a sequence of meshes such that $\delta t^{(m)}$ and $\delta_{\mesh^{(m)}} \to 0$ as $m \to +\infty$~; assume that there exists $\eta >0$ such that $\eta_{\mesh^{(m)}} \le \eta$ for any $m\in \xN$ (with $\eta_{\mesh^{(m)}} $ defined by \eqref{regmesh}); assume moreover that (\ref{infini:boud:h}) and (\ref{infini:bound:u}) hold.
Let $(h^{(m)},\bfu^{(m)})_\mnn$ be a sequence of solutions to the scheme \eqref{scheme} converging to $(\bar h, \bar \bfu)$ in $L^1(\Omega \times (0,T)) \times L^1(\Omega \times (0,T))^2$. 
Then $(\bar h, \bar \bfu)$ satisfies the weak formulation \eqref{sha:eq:pbw} of the shallow water equations. 

If we furthermore assume the space and time steps satisfy \eqref{eq:ent1} and that the sequence of heights is uniformly bounded in $L^1(\Omega,BV)$, \ie\ satisfy \eqref{eq:ent2}, then $(\bar h, \bar \bfu)$ satisfies the entropy inequality \eqref{eq:weakentropy}.
\end{thm}

\begin{proof}
The proof is obtained by passing to the limit in the scheme and in the discrete entropy balance \eqref{eq:ent_disc}, using the tool of \cite{gal-19-wea} (or, more precisely speaking, simplified versions of these tools adapted to Cartesian grids).
The additional assumptions required for the entropy condition are used to prove that the residual term appearing in the discrete potential energy balance, given by \eqref{residu:pot:bound}, tends to zero.
\end{proof}
%
%
\section{Numerical results}

We now assess the behaviour of the scheme on some numerical experiments.
The computations presented here are performed with the CALIF$^3$S free software developed at IRSN \cite{califs}.

\subsection{Rotation in a paraboloid}

This first test case consists in calculating the uniform rotation of a circular drop on a support of parabolic shape (see Figure \ref{fig:drop}).
The computational domain is $(0,L)\times(0,L)$ and the elevation of the support is:
\[
z=-h_0\ \bigl(1-(x- \frac L 2)^2-(y-\frac L 2)^2 \bigr),
\]
with $L=4$ and $h_0=0.1$.
The fluid height is given by
\[
h=h_0\ \max \bigl(0,\ (x -\frac L 2)\,\cos(\omega t) + (y-\frac L 2)\,\sin( \omega t) -z -0.5 \bigr),
\]
and the velocity is
\[
\bfu = \frac 1 2\, \omega\ \begin{bmatrix} -\sin(\omega t) \\ \cos(\omega t) \end{bmatrix}.
\]
It is then easy to check that the mass and momentum balance equations are verified provided that $\omega^2=2\,g\, h_0$.
The solution is thus regular, and this test features a regular topography and dry zones (\ie\ zones where $h=0$).
We compare the numerical and theoretical height obtained after one rotation (\ie\ $\omega t = 2 \pi$), for different uniform grids and with a time step $\delta t=\delta x/8.$ (the maximal speed of sound and the maximal velocity are both close to $1$); results are gathered in the following table.

\medskip
\begin{center}
\begin{tabular}{c | c} \rule[-1.1ex]{0ex}{3.3ex} 
grid & error (discrete $L^1$ norm)
\\ \hline \rule[-1.1ex]{0ex}{3.5ex}
$\quad 100 \times 100 \quad$ &  $3.02\,10^{-3}$
\\ \rule[-1.1ex]{0ex}{3.3ex} 
$\quad 200 \times 200 \quad$ & $1.54\,10^{-3}$
\\ \rule[-1.1ex]{0ex}{3.3ex} 
$\quad 400 \times 400 \quad$ & $0.896\,10^{-3}$
\\ \rule[-1.1ex]{0ex}{3.3ex} 
$\quad 800 \times 800 \quad$ & $0.511\,10^{-3}$
\\ \hline
\end{tabular}
\end{center}

\medskip
We observe an order of convergence between $0.8$ and $1$, which is consistent with a first-order approximation of the fluxes and the time derivative.

\begin{figure}[tp]
  \begin{center} \includegraphics[width=10cm]{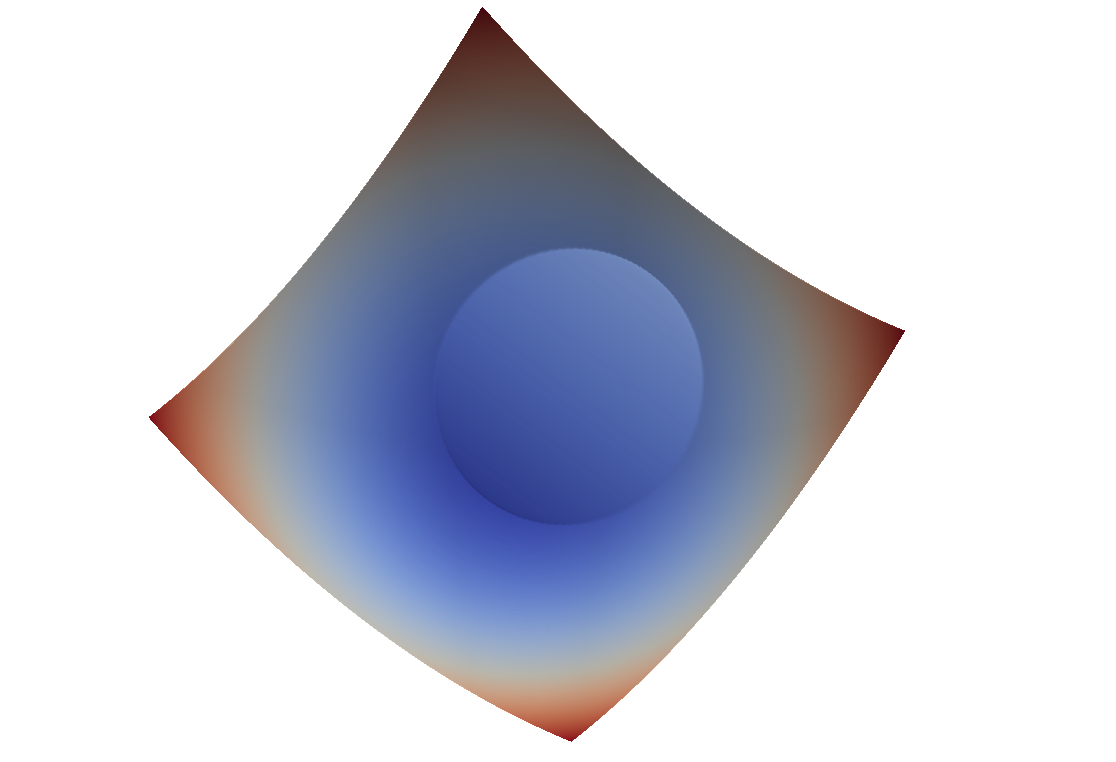} \end{center}
  \caption{Sloshing of a drop on a parabolic support -- State obtained after one revolution (very close to the initial state).}\label{fig:drop}
\end{figure}
%
%
\subsection{A dam-break problem}

In this test, the computational domain is:
\[
\Omega=(0,200)\times(0,200) \setminus \Omega_w \mbox{ with } \Omega_w=(95,105)\times(0,95)\cup (95,17à)\times(0,200).
\]
The fluid is supposed to be initially at rest, and the initial height is $h=10$ for $x_1 \leq 100$ and $h=5$ for $x_1 >100$.
A zero normal velocity is prescribed at all the boundaries of the computational domain.
The computation is performed with a mesh obtained from a $1000 \times 1000$ regular grid, by removing the cells included in $\Omega_w$.
The time step is $\delta t = \delta x /25$ (the maximal speed of sound and the maximal velocity are both close to $10$).
The obtained fluid height is shown at different times on Figure \ref{fig:pdb}; they confirm the efficiency of the scheme, and its capability to deal with reflexion phenomena very simply (\ie\ just by setting the normal velocity at the boundary to zero, by contrast with schemes based on Riemann solvers which need to implement fictitious cells techniques).

\begin{figure}[tp]
  \includegraphics[bb=3cm 2cm 35cm 24cm, clip=true, width=6cm]{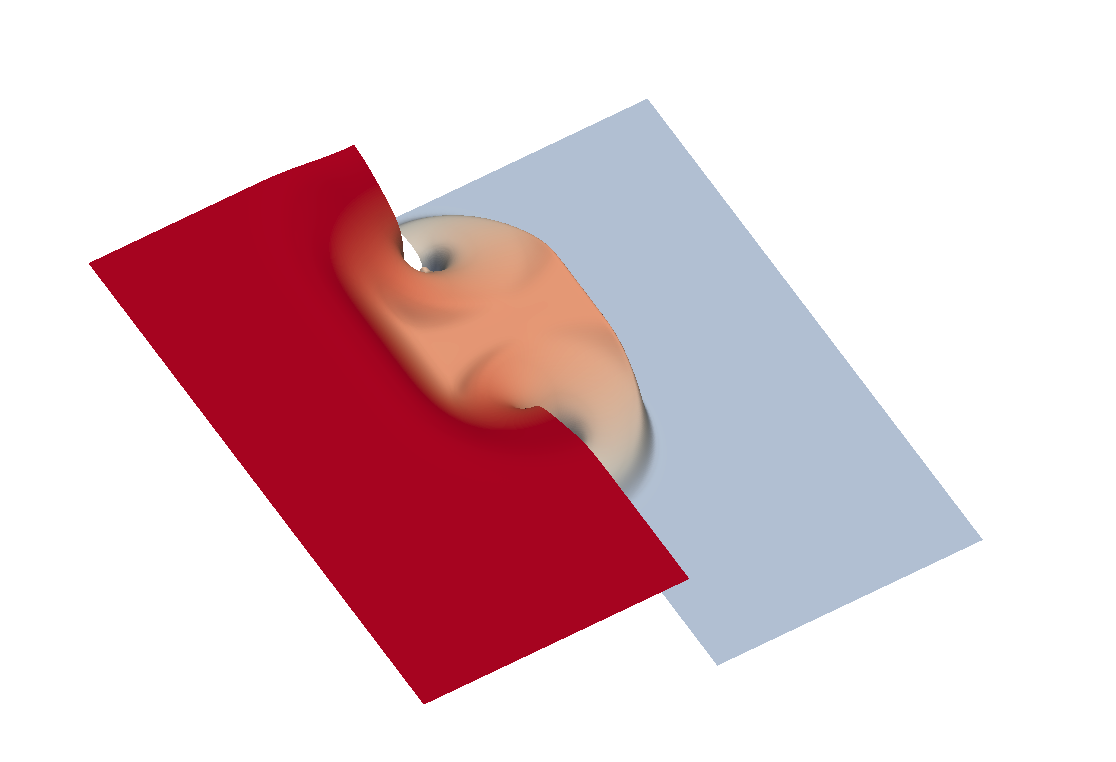}  \includegraphics[bb=3cm 2cm 35cm 24cm, clip=true, width=6cm]{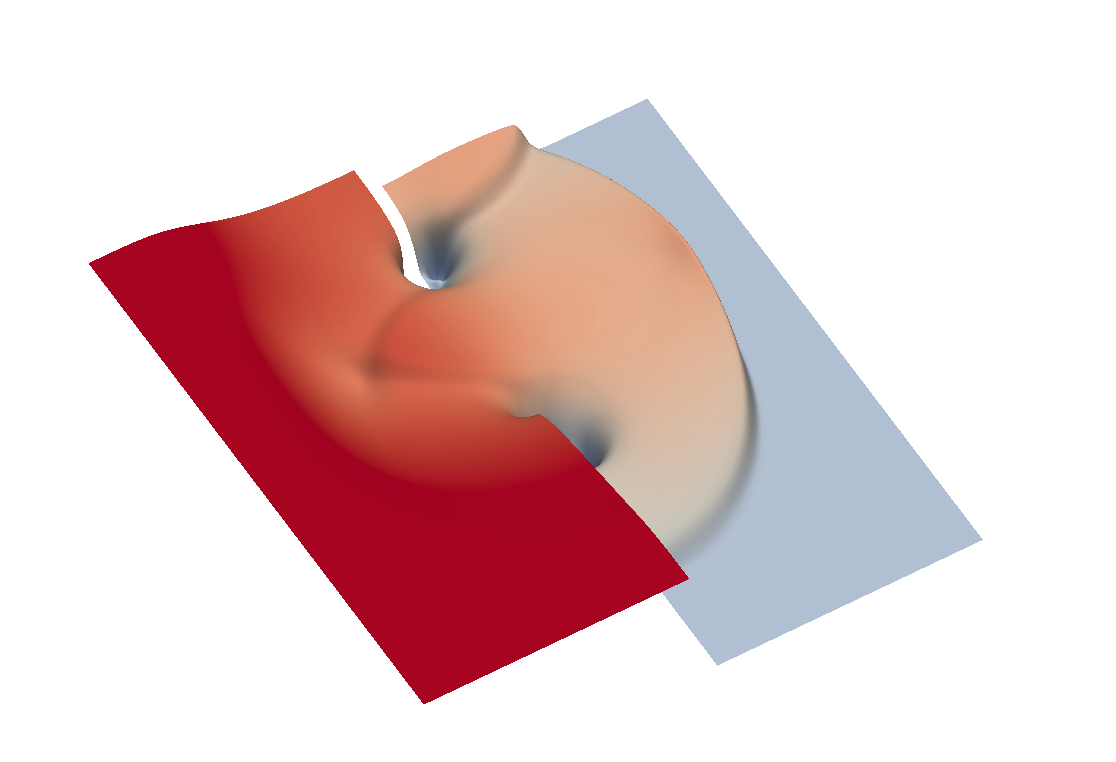}\\
  \includegraphics[bb=3cm 2cm 35cm 26cm, clip=true, width=6cm]{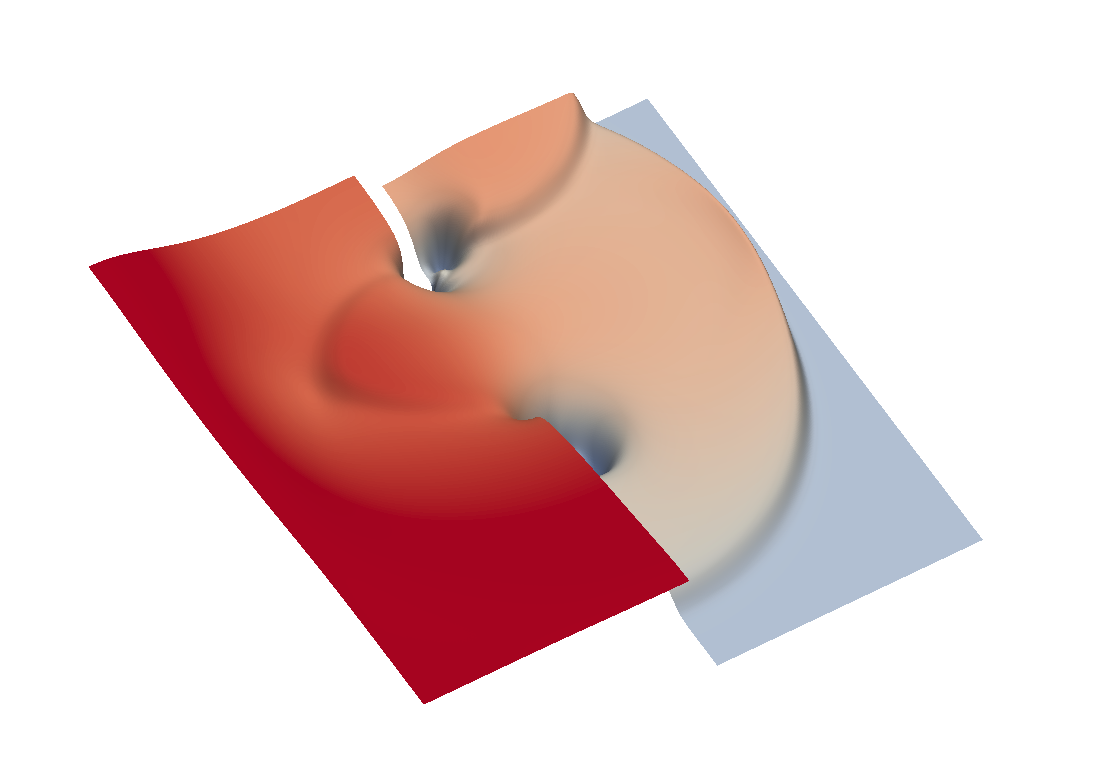} \includegraphics[bb=3cm 2cm 35cm 26cm, clip=true, width=6cm]{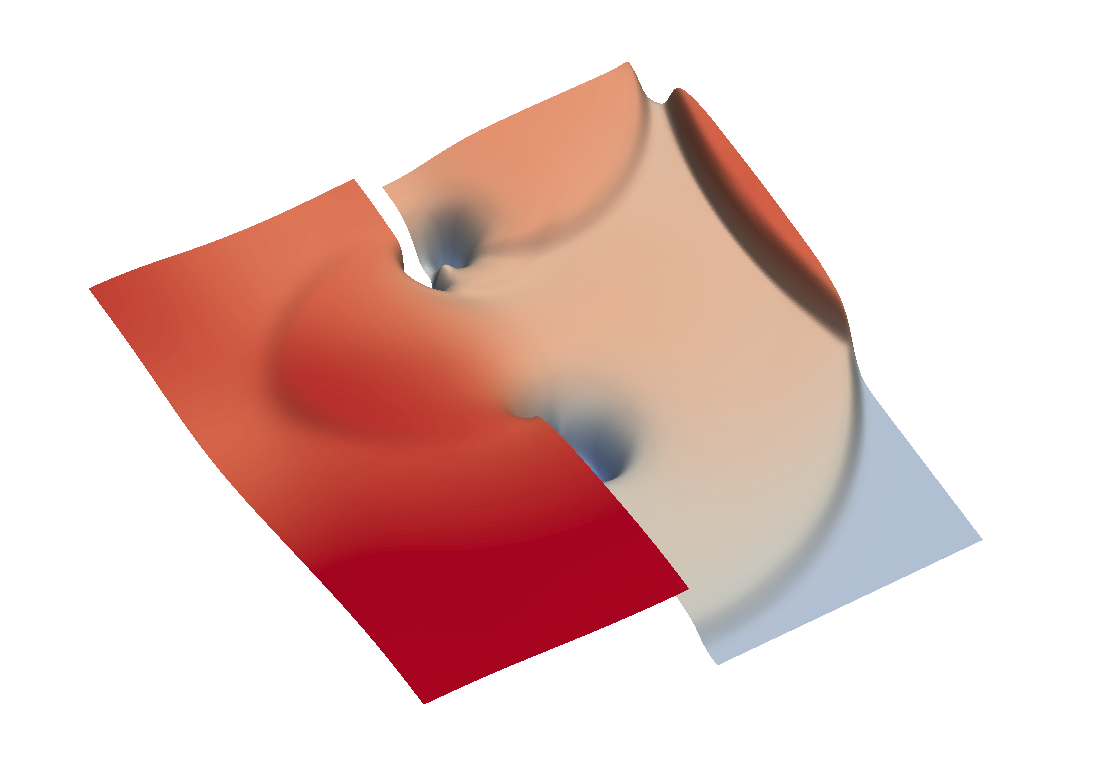}\\
  \includegraphics[bb=3cm 0cm 35cm 27cm, clip=true, width=6cm]{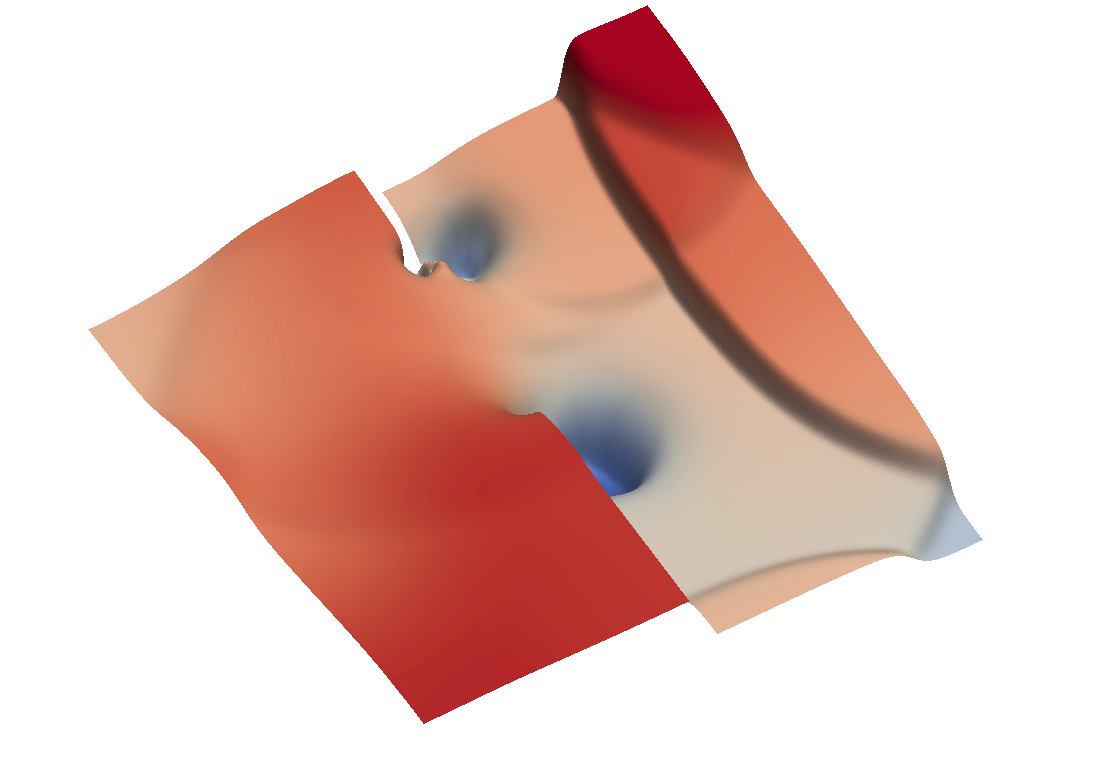} \includegraphics[bb=3cm 0cm 35cm 27cm, clip=true, width=6cm]{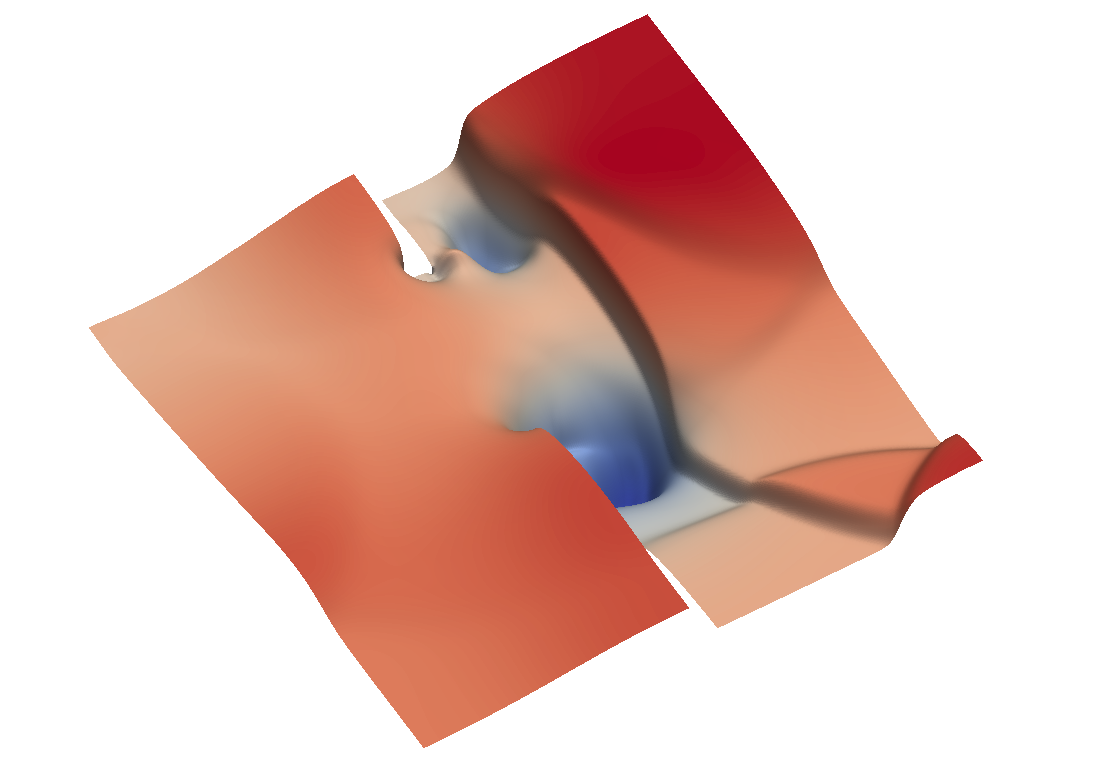}
  \caption{Partial dam break -- Height obtained at $t=4$, $t=8$, $t=10$, $t=12$, $t=16$ and $t=20$ with a mesh obtained (by supression of the zones associated to the obstacles) from a $1000\times 1000$ regular grid. In the last Figure ($t=20$), the obtained minimal and maximal heights are $h=2.149$ and $h=9.306$ respectively.}\label{fig:pdb}
\end{figure}


\ack The authors would like to thank Robert Eymard and Thierry Gallou\"et for several interesting discussions.

\bibliographystyle{acmurl}
\bibliography{sw-jaca}
%
%
\begin{address}
  Rapha\`ele Herbin  and Youssouf Nasseri\\
  Aix-Marseille Universit\'e, Institut de Math\'ematiques de Marseille, \\
  39 rue Joliot Curie\\
  13453 Marseille \\
  \texttt{raphaele.herbin@univ-amu.fr} and \texttt{youssouf.nasseri@univ-amu.fr}
\end{address}
\begin{address}
  Jean-Claude Latch\'e\\
  Institut de Radioprotection et S\^uret\'e Nucl\'eaire,\\
  13115, Saint-Paul-lez-Durance\\
  \texttt{jean-claude.latche@irsn.fr}  
\end{address}
\begin{address}
  Nicolas Therme \\
  CEA/CESTA \\
  33116, Le Barp, France \\
  \texttt{nicolas.therme@cea.fr}
\end{address}
\end{document}